\documentclass[a4paper,11pt]{article}
\usepackage[utf8]{inputenc}
\usepackage[T1]{fontenc}
\usepackage[english]{babel}
\usepackage{amssymb}
\usepackage[dvipsnames]{xcolor}
\usepackage{amsmath}
\usepackage{amsfonts}
\usepackage{comment} 
\usepackage{bbm}
\usepackage{color}
\usepackage{amsthm}
\usepackage{geometry}
\usepackage{fancyhdr}
\usepackage{esint}
\usepackage{authblk}
\usepackage{multicol}
\usepackage{parskip}
\usepackage{url}

\usepackage[mathscr]{euscript}   
\usepackage[color,all]{xy}      
\usepackage{graphicx}
\usepackage[normalem]{ulem}	
\usepackage{tikz}

\usepackage{enumerate}

\newtheorem{remark}{Remark}

\numberwithin{equation}{section}

\author[1]{Nathan Couchet}
\affil[1]{Université Clermont Auvergne, CNRS, LMBP, F-63000 Clermont-Ferrand, France\\ \texttt{nathan.couchet@uca.fr}}

\author[2]{Robert Yuncken \thanks{This article/publication is based upon work from COST Action CaLISTA CA21109 supported by COST (European Cooperation in Science and Technology). www.cost.eu.} }
\affil[2]{Institut Élie Cartan de Lorraine, Université de Lorraine, CNRS, IECL, F-57000 Metz, France \\ \texttt{robert.yuncken@univ-lorraine.fr}}

\title{ A groupoid approach to the Wodzicki residue}
\date{}

\usepackage{xcolor}
\usepackage[colorlinks=true,linkcolor=blue,urlcolor=blue,citecolor=blue]{hyperref}

\newcommand{\C}{\mathbb{C}} 
\newcommand{\R}{\mathbb{R}}

\newcommand{\V}{\mathcal{V}}

\newcommand{\m}{\medbreak} 
\newcommand{\bg}{\bigbreak}
\newcommand{\N}{\mathbb{N}}

\newcommand{\Hn}{\mathbb{H}_{n}}

\newcommand{\Exp}{\mathbb{E}\mathrm{xp}^{\overline{X}} }

\newcommand{\TM}{\mathbb{T} M}

\newcommand{\THM}{\mathbb{T}_HM}
\newcommand{\kgt}{\tilde{\mathbbm{k}}}
\newcommand{\kg}{\mathbbm{k}}

\newtheorem{theoreme}{Theorem}[section]

\newtheorem{lemma}[theoreme]{Lemma}
\newtheorem{corollary}[theoreme]{Corollary}

\theoremstyle{definition}

\newtheorem{definition}[theoreme]{Definition} 

\begin{document}

\maketitle

 

 
\begin{abstract}
Originally, the noncommutative residue was studied in  the 80's by Wodzicki in his thesis \cite{wodzicki1984LocalInvariantSpectralSym} and Guillemin \cite{guillemin1985newProofsWeylsformula}. In this article we give a definition of the Wodzicki residue, using the langage of $r$-fibered distributions from \cite{lescure2017convolution}, \cite{Yuncken2019groupoidapproach}, in the context of filtered manifolds. We show that this groupoidal residue behaves like a trace on the algebra of pseudodifferential operators on filtered manifolds and coincides with the usual residue Wodzicki in the case where the manifold is trivially filtered. Moreover, we show that the groupoidal residue coincides with Ponge's definition \cite{ponge2007residueHeisenberg} for contact and codimension 1 foliation Heisenberg manifolds and Dave-Haller's definition for general filtered manifolds.
\end{abstract}

\section{Introduction.}

One of the remarkable features of the theory of pseudodifferential operators is the noncommutative residue of Wodzicki, which he defined in 1984 in his thesis \cite{wodzicki1984spectralAssymetryNoncommRes}. The noncommutative residue was also defined in 1985 in a article of Guillemin, \cite[Definition 6 p 151]{guillemin1985newProofsWeylsformula}, in which he proposes to associate to an operator $P$ a zeta function:

\begin{equation}
\zeta(P,s) =\sum_{k} \lambda_k^s,
\end{equation}
where $\lambda_k$ are the eigenvalues of $P$, and then shows that it admits a meromorphic continuation. The residues of this zeta function are linked to the number of eigenvalues of $P$ denoted by $N(\lambda)=card \{ \lambda_k, ~ \lambda_k \leq \lambda \}$, leading Guillemin to a Weyl-type formula, as in \cite{weyl1911asymptotische}. 
 
Let $P$ be a classical pseudodifferential operator of ordrer $m \in \mathbb{Z}$ on a manifold $M$ of dimension $d$. This means that in any chart the symbol of $P$ admits an asymptotic expansion: 

\begin{equation}
a(x,\xi) \sim  \sum_k a_{m-k}(x,\xi).
\end{equation}

In the following discussion, we fix a chart $(U,\phi)$ and we identify $U$ with its image in $\R^d$.

\begin{definition}{ \cite[p 58]{vassout2001feuilletages} , \cite[1.8 Formule locale]{kassel1989residu}, \cite[§ 7]{wodzicki1984LocalInvariantSpectralSym}}  \label{751} \m
We define the residue at $x$ of a classical pseudodifferential operator $P$ on $M$ of order $m$ to be:
\begin{equation} \label{766}
Res_x^W(P)=\frac{1}{(2 \pi)^d} \Big( \int_{\mathbb{S}^{d-1}} a_{-d}(x,\xi) d \sigma(\xi) \Big) dx,
\end{equation}
where $a_{-d}(x,\xi)$ is the homogeneous part of order $-d$ in the variable $\xi$ coming from the asymptotic expansion of $P$ in any chart.
Moreover, if $M$ is a compact riemannian manifold, we define also the (global) residue of $P$:
\begin{equation} \label{790}
Res^W(P)=\int_M Res_x^W(P) dx,
\end{equation}
where $dx$ is the smooth measure coming from the riemannian structure.
\end{definition}

Note that the compact hypothesis is to make \eqref{790} valid.  The definition of the residue density $Res_x^W$ in \eqref{766}  makes sense for any $M$.

It is an important theorem that these quantitites are independent of the choice of chart. In this paper, we propose to give another definition of this residue for operators of order $-\dim(M)$, which extends naturally to filtered manifolds, using the groupoidal calculus from \cite{Yuncken2019groupoidapproach} --- that is using tangent groupoids $\TM, \THM$. We therefore manage to circumvent the chart machinery. Moreover our definition extends essentially without change to obtain the noncommutative residues of Ponge \cite{ponge2007residueHeisenberg} and Dave-Haller \cite{DH} of a pseudodifferential operator on any filtered manifold using the filtered tangent groupoid of \cite{choi2019tangent},\cite{VanErp2017groupoid}.

Indeed, the noncommutative residue which we define here can be directly connected with the noncommutative residue of Dave-Haller \cite{DH} by applying results of \cite{couchet2022polyhomo} to  Proposition 2 and Corollary 6 of \cite{DH}.  Thus, the current paper can be seen as a reinterpretation of the Ponge and Dave-Haller noncommutative residues by working directly on the filtered tangent groupoid.  This allows us to realize fundamental properties of the noncommutative residue, notably the invariance with respect to charts and the tracial property, as immediate consequences of the algebraic-geometric structure of the filtered tangent groupoid.  

We remark in passing that the noncommutative residue has proved to be connected with other geometrical objets. For instance :

\begin{enumerate}
\item There is link between this residue for a differential operator on a compact manifold and the asymptotic expansion of the trace of the heat operator $e^{-tP}$, see \cite{ackermann1996noteWodzickiRes, DH}.
\item Connes showed in 1988 - see \cite[Proposition 5 p313]{connes1994noncommutative}, \cite[section 2.6 p17]{cardona2020dixmiertracesdiscretepsido}, \cite[section 7.6]{graciabondia2001Elementsofnoncom}, \cite[Proposition 4.11 p 16]{ponge2021connesweyllaws} - that when $P$ is a pseudodifferential operator of order $-\dim(M)$ and $M$ is compact, then the Dixmier trace of $P$ coincides to this residue up to a constant. 
\item Ponge showed in \cite[Proposition 6.3 p 454]{ponge2007residueHeisenberg} the link between the residue of the Kohn-Laplacian $\square_b$ on a $CR$ compact manifold, see \cite[Equation (6.9)]{ponge2007residueHeisenberg},  and the volume of $M$, defined at \cite[Equation (6.4)]{ponge2007residueHeisenberg}.
\end{enumerate}

Let us very briefly recall the groupoid approach to pseudodifferential operators, first observed by Debord-Skandalis \cite{debord2014adiabatic} in 2014 and developped by van Erp and the second author some years later \cite{Yuncken2019groupoidapproach}. The tangent groupoid of Connes is:
\begin{equation}
\TM=M  \times M \times \R^* \bigcup TM \times \{ 0 \},
\end{equation}
which is seen as a smooth glueing of the tangent bundle $TM$ with a family of pair groupoids $M \times M$ over $\R^*$. In the filtered case the appropriate substitute of the tangent bundle $TM$ is a bundle of nilpotent osculating groups $\mathcal{T}_HM$ whose fibers are denoted by $\mathcal{T}_HM_x$ and we define the filtered tangent groupoid by:
\begin{equation}
\THM=M  \times M \times \R^* \bigcup \mathcal{T}_HM \times \{ 0 \}.
\end{equation}

\bg
The bundle of osculating groups admits a family of automorphism $(\delta_s)_{s >0}$ called dilations which generalises the homotheties on $\TM$ in the trivially-filtered case. 

Using these dilations we obtain a smooth $\R_+^*$-action on $\THM$, namely the Debord-Skandalis action \cite{debord2014adiabatic}, see also \cite{Yuncken2019groupoidapproach} where it is called the \textit{``zoom action''} and \cite{couchet2022polyhomo}:
\begin{definition} \label{799}
Let $M$ be a filtered manifold.
We define the Debord-Skandalis action of $\R_+^*$ on $\THM$, $s \in \R_+^* \mapsto \alpha_s \in Aut(\THM)$ by:
\begin{equation}
\left\{
    \begin{array}{ll}
     \alpha_s(y,x,t)=(y,x,s^{-1}t) ~~ (x,y) \in M,  \\
    \alpha_s(x,\xi,0)=(x,\delta_s(\xi),0) ~~ x \in M, ~ \xi \in\mathcal{T}_HM_x.
    \end{array}
\right.
\end{equation}
\end{definition}

The key point in the groupoidal approach to pseudodifferential operators \cite{Yuncken2019groupoidapproach} is to consider 
the set of distributions on the tangent groupoid $\THM$ which are essentially homogeneous for the Debord-Skandalis action in the following sense. A $r$-fibered distribution on $\THM$ is a continuous $C^\infty(M \times \R)$-linear map  $\kg : C^\infty(\THM) \rightarrow C^\infty(M \times \R)$. This definition implies that the support of such a distribution is $r$-proper and we denote them by $\mathcal{E}_r'(\THM)$. They were first studied by Androulidakis-Skandalis \cite{androulidakis2010pseudodifferential} and Lescure-Manchon-Vassout \cite{lescure2017convolution}  in the general case of submersions. Thanks to a result from \cite{lescure2017convolution}, $r$-fibered distributions can be seen as smooth maps $\kg$ from $M \times \R$ to  
compactly supported distributions in the $r$-fibers, 
$\kg(x,t) \in \mathcal{E}'(r^{-1}(x,t))$, and whose support is $r$-proper. We say that a $r$-fibered distribution is properly supported if its support is also $s$-proper. 

Recall that $C_p^\infty(\THM,\Omega_r)$ denotes the proper smooth sections of the 1-density bundle tangent to the $r$-fibers, see \cite[Definition 5.9]{couchet2022polyhomo}.
We say that a properly supported $r$-fibered distribution $\kg$ is essentially homogeneous of order $m$ for the Debord-Skandalis action if:
\begin{equation} \label{793}
s \in \R_+^* \mapsto s^{-m} \alpha_{s*} \kg - \kg \in  C_p^\infty(\THM,\Omega_r).
\end{equation}
We refer to the function appearing in \eqref{793} as the co-cycles of $\kg$. The set of essentially homogeneous distributions of this kind is denoted 
$\boldsymbol{\Psi}_{\text{vEY}}^m(\THM)$. We will refer to notations and concepts from \cite{Yuncken2019groupoidapproach} and \cite[section 5.3]{couchet2022polyhomo}. Thanks to this, van Erp and the second author define a pseudodifferential operator as the restriction at $t=1$ of an element $\kg \in \boldsymbol{\Psi}_{\text{vEY}}^m(\THM)$,  see \cite[Theorem 2 p3]{Yuncken2019groupoidapproach}. These elements are the $H$-pseudodifferential operators on the filtered manifold $M$ and are denoted by $\boldsymbol{\Psi}_{H}^m(M)$.
\bg
One of the main theorems of the article \cite{Yuncken2019groupoidapproach} is that the groupoidal calculus coincides, in the case of a trivially filtered manifold, with the classical calculus of Kohn-Nirenberg and Hörmander, namely:
\begin{equation} \label{780}
\boldsymbol{\Psi}_{\text{Hör}}^m(M)=\boldsymbol{\Psi}_{H}^m(M).
\end{equation}

A simple but important first observation of the present paper is the following. We will consider $m \in \mathbb{Z}$ in  the whole paper. If $M$ is a filtered manifold, we denote $d_H$ the homogeneous dimension of $M$. 
 Note that, if we fix a smooth measure $dx$ on $M$, we obtain a canonical smooth family of 1-densities $d \lambda_x$ on the tangent fibers $T_xM$ and hence on the osculating groups $\mathcal{T}_HM_x$. 


\begin{lemma} \label{759}
Let $M$ be a filtered manifold and $\kg \in \boldsymbol{\Psi}_{\text{vEY}}^{- d_H}(\THM)$. 
For every $x \in M$, the function defined by:

\begin{equation} \label{805}
s \mapsto \Big( s^{d_H} \alpha_{s*} \kg - \kg \Big)|_{(x,0,0)},
\end{equation}
is a group homomorphism from $(\R_+^*,\times)$ to  $(\C,+)$. More precisely,  there exists a constant $r_x \in \C$ such that for all $s>0$:
\begin{equation}
\Big( s^{d_H} \alpha_{s*} \kg - \kg \Big)|_{(x,0,0)}=r_x \log(s)d \lambda_x.
\end{equation}
\end{lemma}

We shall give the proof in the next section.
\bg
We may define the quantity $r_x dx$ to be the groupoidal residue of the pseudodifferential operator with kernel $\kg|_{t=1}$, which is indeed a polyhomogeneous pseudodifferential operator by \eqref{780}. As we shall see in Lemma \ref{797}, $r_x dx$ does not depend on the $r$-fibered distribution $\kg$ representing the pseudo-differential operator $P$ at $t=1$. Therefore we can define:  


\begin{definition} \label{726}
Let $M$ be a filtered manifold of homogeneous dimension $d_H$ and $P \in \bold{\Psi}_{H}^{m}(M)$ with $m \leq -d_H$. Let $\kg \in \bold{\Psi}_{vEY}^{-d_H}(\THM)$ be an element in vEY groupoidal calculus such that $\kg|_{t=1}$ is the Schwartz kernel of $P$. We define the groupoidal residue density of $P$ at $x \in M$, denoted $Res_x(P)$, to be:
\begin{equation} \label{727}
Res_x(P):= \frac{1 }{\log(s)} \Big( s^{d_H} \alpha_{s*} \kg - \kg \Big)|_{(x,0,0)},
\end{equation}
for any $s \in \R_+^* \setminus \{ 1 \}$.
\end{definition}

This definition should be compared with \cite[Corollary 6(d)]{DH}, where they prove a formula for the residue of a zeta function in terms of the coefficient of degree $-d_H$ in the asymptotic expansion for $P$.
The above definition is for scalar-valued operators. For operators between vector bundles, we should take, for any $s \in \R_+^* \setminus \{ 1 \}$:

\begin{equation} \label{798}
Res_x(P):= \frac{1 }{\log(s)} \mathrm{Tr} \Big( s^{d_H} \alpha_{s*} \kg - \kg \Big)|_{(x,0,0)}.
\end{equation}
For details on the groupoidal calculus with vector-bundle coefficients, see \cite{dave2022gradedHypoellipticity}. We will restrict our attention to scalar-valued operators for simplicity.
\bg
Now come the main results of this paper:
\begin{enumerate}
\item We shall show, see Theorem \ref{746}, that similarly to the Wodzicki residue $Res_x^W$, the groupoidal residue $Res_x$ from Definition \ref{726} defines a trace on operators of appropriate order.
More precisely if $P \in \bold{\Psi}_{H}^{m}(M)$ where $M$ is a filtered manifold and $m \leq -d_H$, $Q \in \bold{\Psi}_{H}^{0}(M)$ then :
\begin{equation} 
Res_x([P,Q]) =0,
\end{equation}
where $[~,~]$ denotes the commutator of operators.
\item  We shall show, see Theorem \ref{758}, that the groupoidal residue from Definition \ref{726} coincides with the Wodzicki residue in the case of a trivially filtered manifold.
\item  We shall show, in Section \ref{filtered_residue}, that the residue density defined of Definition \ref{726} coincides with the definition of the noncommutative residue made by Raphaël Ponge \cite{ponge2007residueHeisenberg} for $\V$-pseudo-differential operators in the calculus of BG on a contact or foliated manifold, \cite[section § 10]{Beals2016Heisenbergcalculus}, and more generally with the definition by Dave-Haller \cite{DH} of the residue for the filtered calculus on any filtered manifold.
\end{enumerate}

Definition \ref{726} only applies to pseudo-differential operators of order $\leq -d_H$. It should be possible to extend this definition to operators of arbitrary order. In the case of the classical unfiltered calculus, this will be treated in a forthcoming article by Higson, Sukochev and Zanin \cite{higson2023NoncommutativeResidue}. We learnt of their work while writing this article. We refer to that article for details. 


\subsection*{Acknowledgements} We would like to thank Dominique Manchon and Sylvie Paycha for their comments and suggestions on this work.

\section{Basic properties of the Wodzicki residue on a filtered manifold.}

Maintaining the notation of the introduction, let us start by proving Lemmas \ref{759} and \ref{797}, which show the well-definedness of our residue. Our conventions for the filtered tangent groupoid are such that the range and source maps are given by :

\begin{equation}
\left\{ \begin{array}{ll}
r(y,x,t)&=y, ~ s(y,x,t)=x, ~ t \neq 0 \\
r(x,v,0)&=x, ~ s(x,v,0)=x.
\end{array}
\right.
\end{equation}

\begin{proof}[Proof of Lemma \eqref{759}]
Set:
\begin{equation} \label{806}
F : s \in \R_+^* \mapsto F_s:=\Big( s^{d_H} \alpha_{s*} \kg - \kg \Big) \in C_p^\infty(\THM,\Omega_r).
\end{equation}
We will show that it is a group morphism when restricted at $(x,0,0)$. First, the reader can easily check that:
\begin{equation} \label{761}
F_{st}=s^{d_H} \alpha_{s*} F_t +F_s, ~ s,t >0.
\end{equation}

Thus we need to show:
\begin{equation} \label{760}
 s^{d_H} \alpha_{s*} F_t|_{(x,0,0)}=F_t|_{(x,0,0)}. 
\end{equation}
We are therefore going to prove \eqref{760}.
In the fiber $(x,.,0)$, we can write:
\begin{equation} \label{804}
 F_s |_{(x,.,0)}=f_x d \lambda_x,
\end{equation}
for some $f_x \in C_c^\infty(\mathcal{T}_HM_x)$ and where $d\lambda_x$ is a Haar measure on the  nilpotent (graded) osculating group $\mathcal{T}_HM_x$ such that $\delta_{s*}(d \lambda_x)=s^{-d_H} d\lambda_x$.  The Debord-Skandalis action at $t=0$ acts by the dilations $\delta_s$.
So for the fiber $(x,.,0)$ we have:
\begin{align}
\alpha_{s*}( F_s )|_{(x,.,0)} &=\delta_{s*}(F_s|_{(x,.,0)}) \\
 & \underbrace{=}_{\eqref{804}} \Big((\delta_s^{-1})^*f_x \Big)  \delta_{s*} (d\lambda_x) \\
& =s^{-d_H} \Big((\delta_s^{-1})^*f_x \Big)  d\lambda_x. \label{794}
\end{align}

Since the Debord-Skandalis action $\alpha_s$ fixes the points $(x,0,0)$, this gives 
$s^{d_H} \alpha_{s*} F_t|_{(x,0,0)}=F_t|_{(x,0,0)}$ as desired. Therefore, $s \mapsto F_s|_{(x,0,0)}$ is a morphism from $(\R_+^*,\times)$ to $(\C,+)$. Moreover, note that the map from \eqref{806} is smooth, see \cite[Lemma 21]{Yuncken2019groupoidapproach}.
\end{proof}


\begin{lemma} \label{797} 
Let $M$ be a filtered manifold and $P \in \boldsymbol{\Psi}_{H}^{m}(M)$, with $m \leq -d_H$.
The Definition \ref{726} of the groupoidal residue does not depend on the $r$-fibered distribution $\kg$ that represents $P$.
\end{lemma}

\begin{proof}
Let $\kg,\kg' \in \bold{\Psi}_{vEY}^{-d_H}(\THM)$ such that:
\begin{equation}
\kg|_{t=1}=P, ~ \kg'|_{t=1}=P.
\end{equation} 
Let $x \in M$ be fixed.
We need to show that the two co-cycles \eqref{727} related to these $r$-fibered distributions give the same groupoidal residue at $x$. Thanks to \cite[Corollary 33]{Yuncken2019groupoidapproach}, we have:
\begin{equation} \label{796}
\kg|_{t=0}-\kg'|_{t=0} \in C_p^\infty(\mathcal{T}_HM,\Omega_r).
\end{equation}
Using \eqref{796}, there exists $f_x \in C_c^\infty(\mathcal{T}_HM_x)$ such that:
\begin{equation}
(\kg-\kg')|_{(x,.,0)}=f_x d\lambda_x,
\end{equation}

where $d\lambda_x$ is a left Haar measure on the nilpotent (graded) Lie group $\mathcal{T}_HM_x$. To conclude, remember that the Debord-Skandalis action fixes $(x,0,0)$. Indeed, we compute:
\begin{align}
\Big( s^{d_H} \alpha_{s*}( \kg-\kg') -( \kg-\kg')  \Big)|_{(x,0,0)} &=  s^{d_H} \alpha_{s*}( \kg-\kg')|_{(x,0,0)} - ( \kg-\kg')|_{(x,0,0)}  \\
& = s^{d_H} \delta_{s*}( f_x d\lambda_x) - f_x d\lambda_x \\ 
& = 0,
\end{align}
as in the previous proof.
\end{proof}

Now  we can conclude that the groupoidal residue is well-defined.
\bg
Let us now recall a basic fact about convolution Lie groups. In the following Lemma the commutators are in the sense of convolution product. Also recall that every connected nilpotent Lie group is unidomular, see \cite[Proposition 2.30]{folland2016harmonicanalysis} or \cite[Proposition 5.5.4 et Corollary 5.5.5]{faraut2008analysis}.

\begin{lemma} \label{757}
Let $G$ be a connected nilpotent Lie group $G$. Given  $f \in C_c^\infty(G)$ and $g \in \mathcal{E}'(G)$ then $[f,g](e)=0$.
\end{lemma}

\begin{proof}
Recall that convolution on Lie group is defined by:
\begin{equation}
\underbrace{u \star v}_{\in ~ \mathcal{E}'(G)}=\underbrace{(u \otimes v)}_{\in ~ \mathcal{E}'(G \times G)} \circ ~ m^*,
\end{equation}
where  $u \in \mathcal{E}'(G), ~ v  \in \mathcal{E}'(G), ~  (m^*f)(x,y)=f(xy).$ We may also take $v \in \mathcal{D}'(G)$ and in this case $u \star v \in \mathcal{D}'(G)$, see  \cite[Theorem 5.1.1 p51 on $\R^d$]{friedlander1998introduction}.
The commutator is well-defined because $C_c^\infty(G)$ is a two-sided ideal in $\mathcal{E}'(G)$. 
Still inspired by the \cite[Theorem 5.2.1 p53 on $\R^d$]{friedlander1998introduction}, one can prove that for $f \in C_c^\infty(G)$ and $g \in \mathcal{E}'(G)$:
\begin{equation}
f \star g(x)= \langle g(y),f(xy^{-1}) \rangle ,
\end{equation}

and:
\begin{equation}
g \star f(x)= \langle g(y),f(y^{-1}x) \rangle.
\end{equation}
See that it is in the two previous equations that we use the unimodularness of $G$.
Then we have proved $f \star g(e)= g \star f(e)$,
that is $[f,g](e)=0$.



\end{proof}

\begin{theoreme} \label{746} \m
Let $M$ be a filtered manifold of homogeneous dimension $d_H$. Let $P \in \Psi_{H}^{m}(M)$  and $Q \in \Psi_{H}^{m'}(M)$ be two pseudodifferential operators $M$ with $m+m' \leq -d_H$. 
Then the groupoidal residue of Definition \ref{726} satisfies the trace property, that is, for all $x \in M$ :
\begin{equation} \label{747}
Res_x([P,Q]) =0.
\end{equation}

\end{theoreme}

\begin{proof} 
We denote by $\kg_P$ and $\kg_Q$ 
the two associated essentially homogeneous  $r$-fibered distributions of order $m$ and  $m'$ respectively. 
By definition, there exists $F,G \in C_p^\infty(\THM,\Omega_r)$ such that:
\begin{equation} \label{1994}
\begin{cases}
\alpha_{s*} \kg_P -s^{m} \kg_P & =F \\
\alpha_{s*} \kg_Q - s^{m'} \kg_Q & =G. 
\end{cases}
\end{equation}
We then compute:

\begin{align}
Res_x([P,Q]) &= Res_x(PQ)-Res_x(QP) \\
             &= \frac{s^{d_H} }{log(s)} \Big( \big( \alpha_{s*} (\kg_P \star \kg_Q) - s^{-d_H} \kg_P \star \kg_Q  \big)|_{(x,0,0)} \nonumber \\  & ~~~ - \big( \alpha_{s*} ( \kg_Q \star \kg_P) - s^{-d_H} \kg_Q \star \kg_P  \big)|_{(x,0,0)}  \Big) \label{792} \\
             & = \frac{s^{d_H}}{log(s)} \Big( \big(( s^{m} \kg_P + F) \star (s^{m'} \kg_Q +G) - s^{- d_H}\kg_P \star \kg_Q  \big)|_{(x,0,0)} \nonumber \\
              & ~~~ -   \big(( s^{m'} \kg_Q + G) \star (s^{m} \kg_P +F) - s^{-d_H}  \kg_Q \star \kg_P \big)|_{(x,0,0)} \Big) \label{789} \\
           & = \frac{s^{d_H}}{log(s)}  \Big( s^m \kg_P \star G +s^{m'} F \star \kg_Q - s^{m'} \kg_Q \star F -s^{m} G \star \kg_P +F \star G- G \star F \Big)|_{(x,0,0)} \\
          &= \frac{s^{d_H}}{log(s)}  \Big( s^{m}[\kg_P,G]+s^{m'}[\kg_Q,F]+[F,G] \Big)|_{(x,0,0)} \\
          & = 0. 
\end{align}

The equality in equation \eqref{789} is true since $\kg_P$ and $\kg_Q$  are essentially homogeneous $r$-fibered distributions respectively of order $m$ and $m'$ respectively and $\alpha_{s^*} : \mathcal{E}_r'(\THM) \rightarrow  \mathcal{E}_r'(\THM)$ is a groupoid automorphism for all $s>0$, see \cite[proposition 3.3.21 p 104]{couchet2023thesis}.
The last equality is true by virtue of Lemma  \ref{757} applied fiberwise to the osculating groups which are by definition connected (graded) nilpotent Lie groups. Indeed the convolution in the brackets $[\kg_P,g]$ and $[\kg_Q,f]$ are done fiberwise. In the fiber at $(x,0)$ the convolution is done between a distribution with compact support and a function on $\mathcal{T}_HM_x$ with compact support. 
\end{proof}

\begin{remark}
Theorem \ref{746} is in particular satisfied for $Q \in \Psi_{H}^{m'}(M)$ with $m'+m \leq -1-d_H$. In this case, we would have $PQ$ and $QP$ in the space $\Psi_{H}^{m+m'-1}(M)$. Therefore, the two $r$-fibered distributions $\kg_P \star \kg_Q$, $\kg_Q \star \kg_P$ would be two continuous functions, thanks to \cite[Theorem 52]{Yuncken2019groupoidapproach} and so the two co-cycles in \eqref{792} would be $0$, as again the Debord-Skandalis action fixes $(x,0,0)$. So it is only the critical case $m+m'=-d_H$ which remained to be proven.  For more details, see \cite[section 7.3 p 161]{couchet2023thesis}. We would like to kindly thank Pr.\ Sylvie Paycha for pointing out to us that Theorem \ref{746} was true for $m+m' \leq -d_H$, and not only for $m \leq -d_H$ and $m'=0$.
\end{remark}


\section{Pseudo-homogeneous functions and kernels}

It is well-known, see for instance \cite{hsiao2008boundaryintegraleq},\cite[p4]{Beals2016Heisenbergcalculus}, that classical pseudodifferential operators admit also a kernel expansion in terms of (pseudo)-homogeneous functions which is precisely linked to the asymptotic expansion of the symbol. Let us recall the details.

Consider $\R^d$ equipped with a one-parameter family of dilations $(\delta_s)_{s>0}$. Typically, these will come from a grading on $\R^d$ where $\delta_s$ acts on the subspace of graded degree $k$ by multiplication by $s^k$. We will be particularly interested in the trivial dilation structure:
 \begin{equation} \label{778}
\delta_s(\xi_1,...,\xi_d)=(s\xi_1,...,s\xi_d),
\end{equation}
and shall also encounter the Heisenberg dilation structure:
\begin{equation} \label{784}
\delta_s(\xi_0,\xi_1,...,\xi_d)=(s^2\xi_0,s\xi_1,...,s\xi_d).
\end{equation}


Now we can define:
\begin{definition}{ \cite[Definition 7.1.1 p353 ]{hsiao2008boundaryintegraleq},\cite[Definition (15.19)]{Beals2016Heisenbergcalculus} } \m
Let $U \subset \R^d$ be an open subset. We denote by $\mathcal{H}_G^m(U \times \R^d)$ the space of functions $f \in C^\infty(U \times \R^d \setminus \{ 0 \})$ homogeneous of order $m$ satisfying:
\begin{equation}
f(x,\delta_s(\xi))=s^m f(x,\delta_s(\xi)), ~ s>0,
\end{equation} 
where $\delta_s$ is as in \eqref{778} or \eqref{784}.
We define the set $\Psi \mathrm{hf}_G^m(U \times \R^d)$ of smooth pseudo-homogeneous functions of degree $m$ in the second variable as follows. If $m \notin \N$ then :
\begin{equation}
\Psi \mathrm{hf}_G^m(U \times \R^d)=\mathcal{H}_G^m(U \times \R^d).
\end{equation}
If $m \in \N$ then $\Psi \mathrm{hf}_G^m(U \times \R^d)$ is the set of $k \in C^\infty(U \times \R^d \setminus \{ 0 \})$ of the form:
\begin{equation}
\label{eq:Psihf}
k(x,\xi) = f(x,\xi)+\log(|\xi|)p(x,\xi),
\end{equation}
where $p$ is a homogeneous polynomial in $\xi$ of degree $m$ having $C^\infty$-coefficients in $x$, where the function $f \in \mathcal{H}_G^m(U \times \R^d)$ and $|~|$ is a homogeneous quasi-norm on $\R^d$, in the sense of dilations \eqref{778} or \eqref{784}.
\end{definition}

We also recall, see \cite{couchet2022polyhomo}:

\begin{definition} 
Let $U \subset \R^d$ be an open subset. We denote by $\mathcal{H} \mathcal{S}_G^m(U \times \R^d)$ the space of functions $f \in C^\infty(U \times \R^d)$ which are homogeneous of order $m$ modulo Schwartz-class meaning:
\begin{equation}
f(x,\delta_s(\xi)) - s^mf(x,\xi),
\end{equation}
is a Schwartz-class function of $\xi$, with smooth dependance on $x$.
\end{definition}

\begin{definition}{ \cite[Equation (7.1.2) p354 ]{hsiao2008boundaryintegraleq} \cite[Definition 3.5 p414]{ponge2007residueHeisenberg}, \cite[Equations (15.40)-(15.41)]{Beals2016Heisenbergcalculus} } \label{765} \m
A distribution kernel $k \in \mathcal{D}'(U \times U)$ is said to have a pseudo-homogeneous expansion of degree $m \in \R$ if:
\begin{equation} \label{773}
k \sim \sum_{j} k_{m+j},
\end{equation}
where $k_{m+j} \in \Psi \mathrm{hf}_G^{m+j}(U \times \R^d)$ and where the symbol $\sim$ means that for all $N \in \N$, there exists $J_N \in \N$:
\begin{equation}
k - \sum_{j \in J_N} k_{m+j} \in C^N(U \times U),
\end{equation}
where $C^N(U \times U)$ denotes the space of class $C^N$ functions. The space of kernels having a pseudo-homogeneous expansion of degree $m$ is denoted  $\Psi \mathrm{hk}_G^m(U)$.
\end{definition}

\begin{remark}
\label{rmk:degree_shift}
Our use of $\mathcal{D}'(U \times U)$ in the above definition follows the standard practice of most works on pseudodifferential operators. Stricly speaking, when we realize these as kernels in the groupoid calculus, we will need to replace these by $r$-fibred distributions :
\begin{equation}
k(x,y)dy \in \mathcal{D}_r'(U \times U),
\end{equation}
see \cite{lescure2017convolution}, \cite{Yuncken2019groupoidapproach}. The Lebesgue measure $dy$ is homogeneous of degree $-d$ with respect to the trivial dilation $\delta_s$ so this introduces a degree-shift in orders of kernels, see the next theorem. Thus, the kernel of a pseudodifferential operator on $U$ of order $m<0$ will be given by an element of $\Psi hk_G^{-m-d}(U)$. This is a well-known technical detail, and we will not remark on it further.
\end{remark}

We now state an important theorem, also see \cite[thm 7.1.1, 7.1.6, 7.1.7, 7.1.8]{hsiao2008boundaryintegraleq}.

\begin{theoreme}{Seeley 1969 \cite[theorem 1 p 209]{seeley2011topics}} \m \label{764}
Let $U \subset \R^d$ be an open subset and $m<0$. Then $P \in \boldsymbol{\Psi}_{\text{Hör}}^m(U)$ if and only if:
\begin{equation}
Pu(x)=\int_U k(x,x-y)u(y)dy, ~ u \in C_c^\infty(U),
\end{equation}
with Schwartz kernel satisfying $k \in \Psi \mathrm{hk}_G^{-m-d}(U)$. 
\end{theoreme}

Moreover, in the case of the trivial dilations structure, the asymptotic expansion of the symbol:

\begin{equation}
a \sim \sum_{j=0}^\infty a_{m-j}, ~ a_{m-j} \in \mathcal{H}_G^{m-j}(U \times \R^d),
\end{equation}

and the kernel:

\begin{equation}
k \sim \sum_{j=0}^\infty k_{m+j}, ~ k_{m+j} \in \Psi hf_G^{m+j}(U \times \R^d),
\end{equation}

are related by an adapted Fourier transform as follows see \cite[Equation (7.1.81) p393]{hsiao2008boundaryintegraleq}. Take  $\psi \in C_c^\infty(\R^d)$ is any cut-off function satisfying:
\begin{equation}
\psi(z):=
\left\{ \begin{array}{lr}
1 & \mbox{if} ~ |z| \leq \frac{1}{2}, \\
0 & \mbox{if} ~ |z| >1.
\end{array} 
\right. 
\end{equation}
Set $\kappa=-m-d$. Then for $m-j <0$:
\begin{equation} \label{776}
a_{m-j}(x,\xi)=\lim\limits_{t \to + \infty} \int_{\R^d} k_{\kappa+j}(x,z) \psi(\frac{z}{t})e^{-i \xi.z} dz, ~ x \in U.
\end{equation}





\section{The Wodziciki residue coincides with the groupoidal residue}

We shall prove in this section that the groupoidal residue $Res_x(P)$ and the Wodzicki residue $Res_x^W(P)$ coincide when $P$ is a classical pseudodifferential operator of order $\leq -d$ on a trivially filtered manifold, see Theorem \ref{758}. We begin by recalling exponential coordinates $\Exp$ on $\TM$, see also \cite{Yuncken2019groupoidapproach}. 
\bg
Given a vector field $X$ on $M$ and a point $x \in M$, we write $\exp(X).x$ for the time one flow of $x$ along $X$ if defined.
If $\overline{X}=(X_1,...,X_n)$ is a local frame of vector fields and  $v \in \R^n$ then we set $v.\overline{X}=\sum_{k=1}^n v_kX_k$. Also, note that the dilations $\delta_s$ on $\R^d$ in this case are given by $\delta_s(v)=sv$.  The following Lemma lists the properties of the exponential charts of $\THM$ which we will need in the sequel.
\bg



\begin{lemma}{\cite[ Lemma 27 p 14]{Yuncken2019groupoidapproach}, \cite[Proposition 5.13]{couchet2022polyhomo}, \cite{couchet2023thesis} } \label{750} \m
Let $M$ be a smooth manifold of dimension $d$ and $\kg \in \boldsymbol{\Psi}_{vEY}^m(\TM)$. Given $x_0 \in M$, $(U_0,\phi)$ a chart on $x_0$ and $\overline{X}=(X_1,...,X_d)$ a local frame on $x_0$, we have: 
\begin{enumerate}
\item There exists an open neighbourhood $U$ of $U_0 \times \{ 0 \}$ with $U \subset U_0 \times \R^d$ such that:
\begin{equation}
Exp^{\overline{X}} : U \rightarrow M \times M, (x,v) \mapsto (x,\exp(v.\overline{X}).x),
\end{equation}
is a diffeomorphism onto its image.
\item The derivative of $Exp^{\overline{X}}$ at $(x_0,0)$ is  : 
\begin{equation}
d_{(x,0)} Exp^{\overline{X}} : (w,v) \in T_xU_0 \times \R^d \mapsto (w, v.\overline{X}|_x).
\end{equation}
\item Put $\tilde{\mathbb{U}}:=\{(x,v,t) \in U_0 \times \R^d \times \R, ~ (x,\delta_t(v)) \in U \}$. Then the map:
\begin{equation}
\Exp : \tilde{\mathbb{U}} \rightarrow \TM, (x,v,t) \mapsto \begin{cases} \Exp(x,v,t)=(Exp^{\overline{X}}(x,\delta_t(-v)),t), &  t \neq 0 \\
\Exp(x,v,0)=(x,v.\overline{X}|_x,0) & t=0,
\end{cases}
\end{equation}
defines  the inverse of a smooth chart for the tangent groupoid $\TM$. 
\item Let  $\mathbb{U}=\Exp(\tilde{\mathbb{U}}) \in \TM$ be the domain of this chart. Then:
\begin{equation}
\mathbb{U}=\Big(TM \times \{ 0 \} \Big) \bigcup \Big(Exp^{\overline{X}}(U) \times \R^* \Big),
\end{equation} 
is an open neighbourhood  of $\Big(TM \times \{ 0 \} \Big) \bigcup \Big(\mathrm{diag}(U_0)\times \R^* \Big)$ in $\TM$, where $\mathrm{diag}(U_0)=\{ (x,x), ~ x \in U_0 \}$. Moreover $\mathbb{U}$ is invariant for the Debord-Skandalis action \ref{799}, and the pullback of this action under $\Exp$:
\begin{equation} \label{772}
\widetilde{\alpha}_{s}:=(\Exp)^{-1} \circ \alpha_s \circ \Exp : \tilde{\mathbb{U}} \rightarrow  \tilde{\mathbb{U}},
\end{equation} 
is given by:
\begin{equation} \label{763}
\widetilde{\alpha}_{s}(x,v,t)=(x,\delta_s(v),s^{-1}t).
\end{equation}

\item There exists $\chi_{\mathbb{U}} \in C_c^\infty(\TM)$ invariant under the Debord-Skandalis action $\alpha_s$ such that:
 \begin{equation} 
\chi_{\mathbb{U}} = \left\{
    \begin{array}{ll}
        1 & \mbox{in a neighbourhood of } \{(x_0,x_0)\} \times \R,  \\
        0 & \mbox{outside } \mathbb{U}.
    \end{array}
\right. 
\end{equation}
\item $\kg \chi_{\mathbb{U}} \in \boldsymbol{\Psi}_{vEY}^m(\TM)$ has support in $\mathbb{U}$. Moreover $\kgt=(\Exp)_*^{-1}(\kg \chi_{\mathbb{U}})$ has support in $\tilde{\mathbb{U}}$ and is essentially homogeneous for the action $\tilde{\alpha}_s$.
\end{enumerate} 
\end{lemma}

\begin{lemma} \label{754}
Let $U \subset \R^d$ be an open subset. If $k_0 \in \Psi \mathrm{h f}_G^0(U \times \R^d)$ then $Res_x(P)=Res_x^W(P)$ is verified for the operator $P \in \bold{\Psi}_{\text{Hör}}^{-d}(U)$ with Schwartz kernel $\chi(x-y) k_0(x,x-y)$ where $\chi \in C_c^\infty(\R^d)$ is $1$ in a neighbourhood of 0 and 0 at infinity.
\end{lemma}

\begin{proof}
By definition of  $\Psi \mathrm{h f}_G^0(U \times \R^d)$, we can write $k_0(x,z)=f_0(x,z)+ \log(|z|)p_0(x)$, where $p_0$ is a smooth function on $U$ and $f_0$ is a smooth function homogeneous of order $0$ with respect to $z$. Let $\chi$ be as in the statement. Then we set:
\begin{enumerate}
\item $l(x,y)=\chi(x-y) \ln(|x-y|)p_0(x)dy$ to be the kernel of the operator $P=Op(l)$ whose kernel's asymptotic expansion is given by $l_0(x,z)=\ln(|z|) p_0(x)dz$ and $l_j(x,z)=0$ if $j>0$.
\item $r(x,y)=\chi(x-y) f_0(x,x-y)dy$ to be the kernel of the operator $P=Op(r)$ whose kernel's asymptotic expansion is given by $r_0(x,z)= f_0(x,z)dz$ and $r_j(x,z)=0$ if $j>0$.
\end{enumerate}

We compute the Wodzicki residue at $x \in U$ respectively for each of these operators, using Definition \ref{766}.
In both cases, recall that equation \eqref{776} gives us the link between the asymptotic symbol expansion and the asymptotic kernel expansion.

\begin{enumerate}
\item We get, by denoting $\mathcal{F}_2$ the Fourier transform with respect to the second variable:
\begin{align} 
a_{-d}(x,\xi) & \underbrace{=}_{\eqref{776}} \mathcal{F}_2 \Big(p_0(x) \log(|.|) \Big)(\xi) \\
&= p_0(x) \mathcal{F}_2( \log(|.|))(\xi), \label{777}
\end{align}
where $\mathcal{F}_2( \log(|.|))$ is interpreted as the Fourier transform of the tempered distribution $z \mapsto \log(|z|)$.
Now, the Fourier transform of the logarithm in $\R^d$ is well known and given, for $\xi \neq 0$ by:
\begin{equation} \label{767}
\mathcal{F}_2( \log(|.|))(\xi)=- \frac{1}{|\xi|^d} \frac{(2 \pi)^d}{\omega_d},
\end{equation}
where $\omega_d=\frac{(2 \pi)^d}{\sqrt{\pi}^d \Gamma(\frac{d}{2}) 2^{d-1}}$ denotes the surface area of the unit $(d-1)$-sphere $\mathbb{S}^{d-1}:=\{\xi \in \R^d, ~ |\xi|=1 \} \subset \R^d$.
Then in \eqref{777} we may now write:
\begin{align}
a_{-d}(x,\xi) & \underbrace{=}_{\eqref{767}} - \frac{p_0(x)}{|\xi|^d} \frac{(2 \pi)^d}{\omega_d}.
\end{align}

Therefore from \eqref{766} we get:
\begin{align}
Res_x^W(Op(l)) & = \frac{1}{(2 \pi)^d} \int_{\mathbb{S}^{d-1}} - \frac{p_0(x)}{|\xi|^d} \frac{(2 \pi)^d}{\omega_d} d \sigma(\xi) dx  = -p_0(x) dx,  \label{786}
\end{align}
where $d \sigma$ denotes the usual surface measure on $\mathbb{S}^{d-1}$.

\item As $f_0 \in \mathcal{H}^0(U \times \R^d)$, we may extend it to $f_0 \in C^\infty(U) \otimes L^\infty(\R^d)$ by attributing any value at $\xi=0$ for all $x \in U$.  The result is a tempered distribution (generalized function) which is homogeneous of degree 0.

We now proceed with $x \in U$ fixed. 
 Thanks to \cite[Proposition 2.4.7 p 140]{grafakos2008classical}, or \cite[p86]{coifman1978deladespsidos},
there exists  $b_x \in \C$ and $\Omega_x$ a smooth function on the sphere $\mathbb{S}^{d-1}$ with integral 0 on $\mathbb{S}^{d-1}$ such that:
\begin{equation} \label{803}
\mathcal{F}_2(f_0(x,.))(\xi)=b_x \delta_0 + W_{\Omega_x}(\xi),
\end{equation}
where $W_{\Omega_x}$ is the principal value distribution whose restriction to $\R^d \setminus \{ 0 \}$ is:
\begin{equation} \label{782}
\Omega_x \Big(\frac{\xi}{|\xi|} \Big) \frac{1}{|\xi|^d},
\end{equation}

see \cite[Equation (2.4.12) ]{grafakos2008classical}.
When $\xi \neq 0$ we have that:
\begin{align}
a_{-d}(x,\xi) & \underbrace{=}_{\eqref{776}}  \mathcal{F}_2(f_0(x,.))(\xi)  \\
& \underbrace{=}_{\eqref{803}} W_{\Omega_x}(\xi), \label{783}
\end{align}
is smooth in $\R^d \setminus \{0 \}$. 
It follows that:
\begin{align}
\int_{\mathbb{S}^{d-1}} a_{-d}(x,\xi) d \sigma(\xi) &  \underbrace{=}_{\eqref{783}} \int_{\mathbb{S}^{d-1}} W_{\Omega_x}(\xi) d \sigma(\xi) \\
& \underbrace{=}_{\eqref{782}} \int_{\mathbb{S}^{d-1}} \Omega_x(\xi) d \sigma(\xi) \\
& = 0,
\end{align}
where the last equality is true by the assumption on $\Omega_x$. We therefore have:
  
\begin{equation} \label{785}
Res_x^W(Op(r))=0.
\end{equation}  
  
\end{enumerate}

Now we look at the co-cycles at $x$ of the $r$-fibred distributions essentially homogeneous associated to the operators $Op(l),Op(r)$ and prove that at $(x,0,0)$ we recover the residue values \eqref{786} and \eqref{785}. First, we can respectively define two elements in $\boldsymbol{\Psi}_{\text{vEY}}^m(\TM)$ such that their restrictions in $t=1$ give the kernels $l$ and $r$.

\begin{enumerate}
\item Set: 
\begin{equation}
\label{eq:cocycle_calculation}
 \left\{
    \begin{array}{ll}
        \mathbbm{l}(x,y,t)= \frac{1}{t^d} \chi(\frac{x-y}{t}) \log(\frac{|x-y|}{t}) p_0(x)dy & \mbox{if } t \neq 0 \\
       \mathbbm{l}(x,v,0)=\chi(v)\log(|v|) p_0(x)d \lambda_x(v) & \mbox{if} ~ t=0,
    \end{array}
\right.
\end{equation}
where $d \lambda_x$ denotes the Haar measure on the tangent space $T_xM$ at $x$.
Writing this in exponential coordinates according to Lemma \ref{750} with respect to the standard coordinate frame $\overline{X}$ for $\R^d$, we get:
\begin{equation} 
\tilde{\mathbbm{l}}(x,v,t)=\Big( \mathbb{E} xp^{\overline{X}}\Big)_*^{-1} \mathbbm{l}(x,v,t)=\chi(v)\log(|v|) p_0(x)d \lambda_x(v).
\end{equation}
Recalling $\tilde{\alpha}_s$ from Lemma \ref{750} and using the fact that $\delta_{s*}(d \lambda_x)=s^{-d}d \lambda_x$, we get:
\begin{align}
s^d \tilde{\alpha}_{s*} \tilde{\mathbbm{l}}(x,v,t) -\tilde{\mathbbm{l}}(x,v,t) &=  s^d \tilde{\mathbbm{l}} \Big(x,\delta_{s^{-1}}(v),st \Big) \delta_{s*} (d \lambda_x(v)) - \tilde{\mathbbm{l}}(x,v,t)  \\
& = - \log(s) \chi(s^{-1}v)p_0(x)d \lambda_x(v) \nonumber  \\
&  ~ ~ ~+ \Big( \chi(s^{-1}v) - \chi(v) \Big)\log(|v|) p_0(x) d \lambda_x(v) \label{800}.
\end{align}
We deduce that :
\begin{equation}
s^d \tilde{\alpha}_{s*} \tilde{\mathbbm{l}} - \tilde{\mathbbm{l}} \in C_p^\infty(\tilde{\mathbb{U}},\Omega_r).
\end{equation}
\item
Set:
\begin{equation} \left\{
    \begin{array}{ll}
        \mathbbm{r}(x,y,t)=  \chi(\frac{x-y}{t}) f_0(x,\frac{x-y}{t})dy & \mbox{if } t \neq 0 \\
       \mathbbm{r}(x,v,0)=\chi(v)f_0(x,v)d \lambda_x(v) & \mbox{if} ~ t=0,
    \end{array}
\right.
\end{equation}
In the same exponential coordinates we get:
\begin{equation} \label{801}
\tilde{\mathbbm{r}}(x,v,t)=\Big( \mathbb{E} xp^{\overline{X}}\Big)_*^{-1}   \mathbbm{r}(x,v,t)=\chi(v) f_0(x,v)d \lambda_x(v),
\end{equation}
and the homogeneity of $f_0$ gives:
\begin{equation}
s^{d} \tilde{\alpha}_{s*}  \tilde{\mathbbm{r}} -\tilde{\mathbbm{r}} \in C_p^\infty(\tilde{\mathbb{U}},\Omega_r).
\end{equation}
\end{enumerate}
We move now to the computations of the co-cycles restricted in $(x,0,0)$. 

\begin{enumerate}
\item Using \eqref{800} we get, for all $s \in \R_+^* \setminus \{ 1\}$:
\begin{align}
\Big( \tilde{\alpha}_{s*} \tilde{\mathbbm{l}} -s^{-d} \tilde{\mathbbm{l}} \Big)|_{(x,0,0)} & = - \frac{1}{s^d} \log(s) p_0(x) d \lambda_x \\
&=  \frac{1}{s^d} log(s)  Res_x^W(Op(l)),
\end{align}
where we use the canonical identification of the smooth family of 1-densities $d \lambda_x$ with the smooth measure $dx$ on $M$.
Then:
\begin{equation}
 \frac{1}{\log(s)} \Big(s^d \tilde{\alpha}_{s*}\tilde{\mathbbm{l}} -  \tilde{\mathbbm{l}}  \Big)|_{(x,0,0)} = Res_x^W(Op(l)).
\end{equation}
\item Using \eqref{801} we get, for all $s \in \R_+^* \setminus \{ 1\}$:
\begin{align}
\Big( \tilde{\alpha}_{s*} \tilde{\mathbbm{r}} -s^{-d} \tilde{\mathbbm{r}} \Big)|_{(x,0,0)} &=\Big( f_0(x,v) [ \chi(\frac{v}{s}) - \chi(v) ] d \lambda_x(v) \Big)|_{(x,0,0)} \\
&= 0 \\
& =  \frac{1}{s^d} \log(s)  Res_x^W(Op(r)).
\end{align}

Then:

\begin{equation}
 \frac{1}{\log(s)} \Big( s^d \tilde{\alpha}_{s*}\tilde{\mathbbm{r}} - \tilde{\mathbbm{r}} \Big)|_{(x,0,0)} = Res_x^W(Op(r)).
\end{equation}

\end{enumerate}
This completes the proof.
\end{proof}


\begin{theoreme} \label{758} \m
Let $M$ be a (trivially) filtered manifold of dimension $d$ and $P \in \boldsymbol{\Psi}_{H}^m(M)$  a classical pseudodifferential operator of order $m$ on $M$, with $m \leq -d$, $m \in \mathbb{Z}$. Let $\kg$ be any essentially homogeneous $r$-fibered distribution of order $-d$ that extends $P$ at $t=1$. Then:
\begin{equation} \label{753} 
 Res_x^W(P) =\frac{1}{\log(s)} \Big( s^d \alpha_{s*} \kg - \kg \Big)|_{(x,0,0)} , ~ \forall ~ s \in \R_+^* \setminus \{ 1\}, ~ \forall ~ x \in M.
\end{equation}
\end{theoreme}

\begin{proof}
Consider first the case  $m \leq -d-1$.
Thanks to  \cite[Theorem 52]{Yuncken2019groupoidapproach} we already know that:
\begin{equation} \label{795}
\kg \in C^0(\TM,\Omega_r),
\end{equation}
as we have supposed here $m \leq -d-1$. That means:
\begin{equation}
\kg(x,v,0)=\mathbbm{l}_0(x,v) d \lambda_x,
\end{equation}
for some $\mathbbm{l}_0 \in C^0(TM)$ and where $d \lambda_x$ is the Haar measure on $T_xM$.





We can now evaluate the co-cycle term to term. Moreover the points $(x,0,0)$ are fixed by the Debord-Skandalis action. 
Thus, using the facts that $\delta_{s*} d \lambda_x=s^{-d} d \lambda_x$,
we get:
\begin{align}
\frac{1}{\log(s)} \Big(s^d \tilde{\alpha}_{s*} \kg - \kg \Big)|_{(x,0,0)} & = \frac{1}{\log(s)} \Big(  \mathbbm{l}_0(x,0) d \lambda_x - \mathbbm{l}_0(x,0) d \lambda_x \Big) = 0.
\end{align}
This agrees with the Wodzicki residue in this case. Indeed, for an operator of this order the term $a_{-d}$ appearing in the asymptotic expansion of its symbol is always zero.
\bg
If $m=-d$ and $(U_0,\phi)$ is a chart on $x \in M$,
then the kernel admits an asymptotic expansion  $k \sim \sum_{j=0}^{+ \infty} k_{j}$ in $U_0$, thanks to Seeley's Theorem \ref{764}. Thus we may write:
\begin{equation} \label{802}
k(x,x-y) - \chi(x-y) k_0(x,x-y) \sim \sum_{j \geq 1} k_j(x,x-y),
\end{equation}
with $\chi \in C_c^\infty(\R^d)$ is such that $\chi$ is equal to 1 in a neighbourhood of $0$ and $0$ at infinity. Let us denote the left hand side of \eqref{802} by $\overline{k}(x,x-y)$. Then $\overline{P}=Op(\overline{k}) \in \bold{\Psi}_{\text{Hör}}^{m-1}(M)$ and using what we did just before, $Res_x^W(\overline{P})=Res_x(\overline{P})$ for all $x \in M$.  It suffices to apply Lemma \ref{754} to the function  $k_0 \in \Psi \mathrm{hf}_G^0$ to conclude, as we have $k(x,z)=\overline{k}(x,z)+\chi(z)k_0(x,z).$
\end{proof}

\section{The noncommutative residue on a filtered manifold}
\label{filtered_residue}

In his article \cite{ponge2007residueHeisenberg}, Ponge defined a noncommutative residue that fits the context of a Heisenberg manifold.
This was generalized to arbitrary filtered manifolds by Dave and Haller in \cite{DH}.
In this section, we will show that their definitions coincide with the groupoidal residue of Definition \ref{726} for pseudodifferential operators of order $\leq -d_H$.
 Again, we will restrict our attention to scalar-valued operators to simplify notation, although one can easily generalise to vector bundles using \eqref{798}.

Let us begin with Ponge's noncommutative residue for a Heisenberg manifold.
Let $M$ be a Heisenberg manifold of dimension $d+1$ with hyperplane bundle $\V \leq TM$. The algebra of Heisenberg pseudodifferential operators of BG \cite{Beals2016Heisenbergcalculus} is denoted $\bold{\Psi}_\V^{\bullet}(M)$. 
It is shown in \cite{couchet2022polyhomo} that this coincides with the groupoidal calculus when $M=\Hn \times \R^m$, $\Hn$ being the $2n+1$ dimensional Heisenberg group, or $M$ is a contact manifold or a codimensional one foliation. That is:
\begin{equation}
\bold{\Psi}_\V^{m}(M)=\bold{\Psi}_{H}^{m}(M),
\end{equation}
for $M$ in these cases, though we expect it to be true also for a general Heisenberg manifold.
\bg 
Ponge's noncommutative residue is defined as follows.
Let $\Big( X_j \Big)_{j \in \{ 0,...,d \}}$ be a local $H$-frame of vector fields on an open subset $U \subset M$ and $\Psi_x : U \rightarrow \R^{d+1}$ be a privileged change of coordinates centered at $x$, see \cite[p 415 and Definitions 2.3 ,2.4]{ponge2007residueHeisenberg}. The latter assertion means that if $\Big( X_j \Big)_{j \in \{0,...,d \}}$ is a local $H$-frame of vector fields, then we have $\Psi_x(x)=0$ and $(\Psi_x)_*X_j(x)=\partial_j|_x$. In this context, the noncommutative residue of Ponge of a Heisenberg pseudodifferential operator $P \in \bold{\Psi}_{\V}^{-d_H}(U)$ of degree $-d_H$ on $U$, is defined as follows.
\bg
Let $p_{-d-2}$ be the term of degree $-d_H=-(d+2)$ from the asymptotic expansion of the symbol of $P \in \bold{\Psi}_{\V}^{-d_H}(U)$, see \cite[2.7 p409]{ponge2007residueHeisenberg}. Then set the noncommutative residue of $P$ at $x$: 
\begin{equation}
c_P(x)=\frac{|d \Psi_x|}{(2 \pi)^{d+1}} \int_{\mathbb{S}^{d}} p_{-(d+2)}(x,\xi) d\xi,
\end{equation}
where $|d \Psi_x|$ is the jacobian of $\Psi_x$, see \cite[Lemma 3.9]{ponge2007residueHeisenberg}.
The reader can compare this definition in contrast with the "non-graded" non commutative residue \eqref{766}.

We will first give a direct proof of the equality of the groupoidal residue and Ponge's residue on a contact manifold or foliation.  Note that we will give a more general result afterwards in Theorem \ref{thm:Dave-Haller-residue} by taking advantage of results of Dave-Haller \cite{DH}.

We begin with the case of the model groups $M=\Hn \times \R^m=\R^{d+1}$, where $d=2n+m$. If $n=0$ then $\mathbb{H}_0=\R$ by convention. We equip $M$ with the model vector fields $\overline{X}=(X_0,X_1,...,X_d)$ of \cite[chapter 1 p12-13]{Beals2016Heisenbergcalculus}, also see \cite[section 5.1]{couchet2022polyhomo}, so that $(X_0,...,X_{2n})$ generate $\Hn$, $X_0$ is central and $(X_{2n+1},...,X_d)$ are the usual vectors fields on $\R^{d+1}$. Recall that $\mathcal{E}_r'(M \times M)$ denotes the set of $r$-fibered distributions on the pair groupoid $M \times M$.

\begin{theoreme} \label{774}
Given a model manifold $M=\Hn \times \R^m$ of homogeneous dimension $d_H=d+2$, $2n+m=d$, with the standard model structure as in \cite{couchet2022polyhomo} and $P \in \bold{\Psi}_{\V}^{m}(M)$, with $m \leq -d_H$, then the residue of Ponge at $x$ and our residue from Definition \ref{726} coincide:
\begin{equation} \label{781}
Res_x(P)=c_P(x).
\end{equation}
\end{theoreme}

\begin{proof}
We will rely heavily on \cite{Beals2016Heisenbergcalculus} and \cite{couchet2022polyhomo}. We denote $k \in \mathcal{E}_r'(M \times M)$ the kernel of $P$. Since the Heisenberg and groupoidal calculi coincide, see \cite[Theorem 5.16]{couchet2022polyhomo}, there exists $\kg \in \bold{\Psi}_{vEY}^{-d_H}(\THM)$ such that $\kg|_{t=1}=k$. Moreover we may suppose, see \cite[proposition 42]{Yuncken2019groupoidapproach} that $\kg$ is homogeneous on the nose of order $-d_H$ outside $t \in [-1,1]$. 
We pull this back via exponential coordinates, again as in \cite[section 5]{couchet2022polyhomo}, yielding $\kgt \in \mathcal{E}_r'(M \times \R^{d+1} \times \R)$ with the equality $\kg=\Exp_* \Big( \kgt \Big)$. Note that in the case of model manifold, the exponential coordinate chart is globally define.
\bg
Next we must consider the symbol of $P$. By the definition of the Heisenberg calculus, \cite[Chapter 3, § 10]{Beals2016Heisenbergcalculus}, $P$ is defined starting from a graded-polyhomogeneous function $f \in S_{phg,G}^{-d_H}(M \times \R^{d+1})$. Letting $\overline{\sigma}(x,\xi)=(x,\sigma_0(x,\xi),...,\sigma_d(x,\xi))$ be the coordinate transform \cite[Equations (10.14),(10.15)]{Beals2016Heisenbergcalculus} obtained from the symbols of the model vector fields $(X_0,...,X_d)$, Beals and Greiner define the $\V$-symbol associated to $f$ by:
\begin{equation}
q(x,\xi)=\overline{\sigma}^*f(x,\xi).
\end{equation}
See also \cite[Definition 5.5]{couchet2022polyhomo}.
The symbol and kernel are related by fiberwise Fourier transform, after the abovementioned coordinate changes. Explicitly, we have:
\begin{equation}
f=\mathcal{F}_2(\kgt)|_{t=1},
\end{equation}
where $\mathcal{F}_2$ is the fiberwise Fourier transform in the second variable. Extending this fiberwise Fourier transform to all $t \in \R$, let us put:
\begin{equation}
u=\mathcal{F}_2(\kgt) \in C^\infty(M \times \R^{d+1} \times \R).
\end{equation}
Note that the Debord-Skandalis action transforms under the Fourier transform as:
\begin{equation}
\mathcal{F}_2 \circ \tilde{\alpha}_{s*} \circ \mathcal{F}_2^{-1}=\beta_s^*,
\end{equation}
where $\beta_s : M \times \R^{d+1} \times \R \rightarrow M \times \R^{d+1} \times \R$ are the dilations:
\begin{equation}
\beta_s(x,v,t)=(x,\delta_s(v),st),
\end{equation}
see \cite[Proposition .13]{couchet2022polyhomo}. The essential homogeneity of $\kg$ and consequently of $\kgt=(\Exp_*)^{-1}(\kg)$, therefore implies that $u \in \mathcal{H}\mathcal{S}_G^m(M \times \R^{d+1} \times \R)$ where the homogeneity modulo Schwartz is with respect to the dilations $\beta_s$ (thanks to the hypothesis on the homogeneity of $\kg$ outside $t \in [-1,1]$), see the proof of \cite[Theorem 5.16]{couchet2022polyhomo}. 
\bg
Set $u_0=u|_{t=0}$. By \cite[Proposition 3.2]{couchet2022polyhomo} we have $u_0 \in \mathcal{H} \mathcal{S}_G^{-d_H}(M \times \R^{d+1})$. Therefore by a well-known Lemma, eg \cite[Theorem 2.1]{couchet2022polyhomo}, outside a compact neighbourhood containing $0$ of $\xi$ we may write:
\begin{equation}
u_0=u_0'+u_0'',
\end{equation}
where $u_0' \in \mathcal{H}_G^{-d_H}(M \times \R^{d+1})$ and $u_0'' \in S_G(M \times \R^{d+1})$. 
\bg
We may extend the homogeneous function $u_0'$ as a tempered distribution (not necessarily homogeneous), still denoted $u_0'$. Indeed, we may first extend $u_0'$ to a distribution such as in \cite[section § 15]{Beals2016Heisenbergcalculus}, \cite[Lemma 3.1]{ponge2008heisenberg}, or in \cite[Theorem 3.2.4]{hormander1983analysis},
and it is tempered because it has polynomial growth at infinity. 
Therefore we can find $u_0''' \in C^\infty(M,\mathcal{E}'(\R^{d+1}))=C^\infty(M) \otimes \mathcal{E}'(\R^{d+1})$ such that the following holds everywhere:
\begin{equation} \label{769}
u_0=u_0'+u_0''+u_0'''.
\end{equation}
We can compute for all $s \in \R_+^* \setminus \{ 1 \}$:
\begin{align}
Res_x(P)&= \frac{1 }{\log(s)} \Big( s^{d_H} \alpha_{s*} \kg - \kg \Big)|_{(x,0,0)} \\
& =\frac{1 }{\log(s)} \Big( s^{d_H} \alpha_{s*} \Exp_* \Big( \kgt \Big) - \Exp_* \Big( \kgt \Big) \Big)|_{(x,0,0)} \label{771} \\ 
&=\frac{1 }{\log(s)} \Big( s^{d_H} \tilde{\alpha}_{s*} \kgt - \kgt \Big) \circ \Big( \Exp \Big)^{-1} |_{(x,0,0)}, \label{768}
\end{align}
where we used in \eqref{771} the equality $\tilde{\alpha}_s=\Big( \Exp \Big)^{-1} \circ \alpha_s \circ \Exp$, 
see \eqref{772}.
Also recall that $\Exp|_{M \times \{ 0 \} \times \{ 0 \}}=id_{M \times \{ 0 \} \times \{ 0 \}}$.
We next continue to compute in \eqref{768}:
\begin{align}
Res_x(P)&=\frac{1 }{\log(s)} \Big( s^{d_H} \tilde{\alpha}_{s*} \mathcal{F}_2^{-1}(u) - \mathcal{F}_2^{-1}(u) \Big)|_{(x,0,0)} \\
&= \frac{1 }{\log(s)} \mathcal{F}_2^{-1} \Big(s^{d_H} \beta_s^*u -u \Big)|_{(x,0,0)},
\end{align}
where we use the equality $\mathcal{F}_2 \circ \tilde{\alpha}_{s*} \mathcal{F}_2^{-1}=\beta_s^*$ recalled earlier. Now we use \eqref{769}:
\begin{align}
Res_x(P)
&= \frac{1 }{\log(s)} \mathcal{F}_2^{-1} \Big(s^{d_H} \beta_s^*(u_0'+u_0''+u_0''') -(u_0'+u_0''+u_0''') \Big)|_{(x,0,0)} \\
&= \frac{1 }{\log(s)} \mathcal{F}_2^{-1} \Big(s^{d_H} \beta_s^*u_0' -u_0'+ s^{d_H} \beta_s^*(u_0''+u_0''') -(u_0''+u_0''')   \Big)|_{(x,0,0)}.
\end{align}

Since $u_0''$ is Schwartz class in $\xi$ and $u_0'''$ is compactly supported in $\xi$, their fiberwise Fourier transforms are smooth. 

Therefore $\mathcal{F}_2^{-1}(u_0'')$ and $\mathcal{F}_2^{-1}(u_0''')$, 
can be evaluated at $(x,0,0)$ and as in the proof of Theorem \ref{758}, for all $s \in \R_+^* \setminus \{ 1 \}$, we get:
\begin{equation}
 \frac{1 }{\log(s)} \mathcal{F}_2^{-1} \Big(s^{d_H} \beta_s^*(u_0''+u_0''') -(u_0''+u_0''')   \Big)|_{(x,0,0)}=0.
\end{equation}
Hence:
\begin{equation} \label{770}
Res_x(P) = \frac{1 }{\log(s)} \mathcal{F}_2^{-1} \Big(s^{d_H} \beta_s^*u_0' -u_0' \Big)|_{(x,0,0)}. 
\end{equation}

Finally, we use \cite[Lemma 3.1 p 414 and Equation (3.2)]{ponge2007residueHeisenberg} which assert that we have for all  $s \in \R_+^* \setminus \{ 1 \}$:
\begin{equation} \label{807}
\beta_s^* u_0'=s^{-d_H}u'_0+s^{-d_H} \log(s) c_0(u_0') \delta_0,
\end{equation}
where:
\begin{equation}
c_0(u_0')=\int_{|\xi|=1} u_0'(x,\xi) d \sigma(\xi).
\end{equation}

Then, using \eqref{807}, Equation \eqref{770} becomes:
\begin{align}
Res_x(P)&= \frac{1 }{\log(s)} \mathcal{F}_2^{-1} \Big(\log(s) c_0(u_0') \delta_0 \Big)|_{(x,0,0)} \\
& =\frac{1}{(2 \pi)^{d+1}}  \int_{|\xi|=1} u_0'(x,\xi) d \sigma(\xi),
\end{align}
where the constant $\frac{1}{(2 \pi)^{d+1}}$ appears in the inverse Fourier transform formula. 
Now, thanks to \cite[Equation (3.25) p 27]{Beals2016Heisenbergcalculus} we see that:
\begin{equation}
^t \Big( d \Psi_x  \Big) \xi=\sigma(x,\xi),
\end{equation}
where $\sigma$ is defined in \cite[Equation (5.12)]{couchet2022polyhomo}. The reader can compute the Jacobian of $\Psi_x$ and see that it is triangular and unipotent and so has determinant equal to one, see also \cite{couchet2023thesis}.
 
According to \cite[the proof p8 of Theorem 1.12]{couchet2022polyhomo}, the purely homogeneous component $u_0'$ of $u_0$ is the first term in the asymptotic expansion of the polyhomogeneous function $f=u|_{t=1}$, see \cite[Equation (3.16)]{couchet2022polyhomo} and the remarks which follow. Therefore:
\begin{equation}
u_0'(x,\xi)=p_{-(d+2)}(x,\xi),
\end{equation}
where $p_{-(d+2)}$ is the term of degree $-(d+2)$ in the asymptotic expansion of $f=u|_{t=1}$. 
We have shown that :
\begin{equation}
Res_x(P)=\frac{|d \Psi_x|}{(2 \pi)^{d+1}}  \int_{|\xi|=1} p_{-(d+2)}(x,\xi) d \sigma(\xi).
\end{equation}
\end{proof}

\begin{corollary} \label{779}
If $M$ is a contact manifold or a codimensional one foliation, then the groupoidal residue of Definition \ref{726} agree with's Ponge noncommutative residue for operators $P \in \bold{\Psi}_\V^{-d_H}(M)$, meaning that \eqref{781} still holds.
\end{corollary}

\begin{proof}
Darboux' Theorem for a contact manifold or Frobenius' Theorem for a codimensional one foliation, implies that around any $x \in M$, there is a local coordinate system which identifies $M$ with the model 
space $\Hn$, $d=2n$, or $\mathbb{H}_0 \times \R^d$. Since both Ponge's residue and the groupoidal residue are independent of (priveleged) coordinates, the result follows.
\end{proof}

We conclude by proving the equality of the groupoidal residue with Dave and Haller's noncommutative residue.  Let $M$ be a closed filtered manifold.  Recall that in this context, the appropriate analogue of an elliptic operator is a Rockland operator, as follows.  At each point $x\in M$, a differential operator $D$ of filtered order $m$ has a principal part $\sigma^m_x(D)$ which is a left-invariant homogeneous differential operator on the osculating group $\mathcal{T}_HM_x$ and so an element of the enveloping algebra $\mathcal{U}(\mathfrak{t}_HM_x)$.  The operator is Rockland if for every $x\in M$ and every non trivial irreducible unitary representation of $\mathcal{T}_HM_x$, the operator $\pi(\sigma^m_x(D))$ is invertible.

Fix a strictly positive Rockland differential operator $D$ on $M$ of even order $r>0$ with respect to the filtration.  If $P\in\Psi^m_H(M)$ is a pseudodifferential operator in the groupoidal calculus of \cite{Yuncken2019groupoidapproach}, then Dave and Haller \cite{DH} prove that the zeta function
\[
  \zeta_P(z) = \mathrm{Tr}(PD^{-z})
\]
is well-defined for $\Re(z)>(n+k)/r$ and admits a meromorphic extension with at most simple poles located at $(d_H+m-j)/r$ with $j\in\N$.  They then define a noncommutative residue $\tau$ by
\begin{equation}
\label{eq:Dave-Haller-residue}
  \tau(P) = r.\mathrm{res}_{z=0}(\zeta_P(z)).
\end{equation}

\begin{theoreme}
\label{thm:Dave-Haller-residue}
Let $M$ be a closed filtered manifold and $P\in\Psi^m_H(M)$.  If the order of $P$ is $m\leq-d_H$, then the noncommutative residue \eqref{eq:Dave-Haller-residue} of Dave-Haller is equal to the integral of the groupoidal residue of Definition \ref{726},
\[
  \tau(P) = \int_M Res_x(P).
\]
\end{theoreme}

\begin{proof}
Let $k\in\mathcal{E}'_r(M\times M)$ be the kernel of $P$.  It suffices to consider the case when $m=-d_H$, since otherwise both residues are $0$.  Thus there is $\kg \in \mathcal{E}'_r(\mathbb{T}_HM)$ essentially homogeneous of order $m=-d_H$ for the Debord-Skandalis action, such that $\kg|_{t=1} = k$.   The restriction $k_0 := \kg|_{t=0} \in \mathcal{E}'_r(\mathcal{T}_HM)$ is a fibred distribution which is homogeneous of order $-d_H$ modulo $C^\infty$, with respect to the pushforward of distributions by the graded dilations $\delta_s$.  Specifically, it is the first term in the pseudo-homogeneous expansion of the kernel $k$, see \emph{e.g.} \cite[Remark 3.4]{DH:BGG}.  Interpreting $k_0$ as a generalized function,
it is pseudohomogeneous function of degree $0$, with the degree shift of $d_H$ as explained in Remark \ref{rmk:degree_shift}.  Explicitly, this means that 
\[
  k_0(x,\xi) = f_0(x,\xi) + \log(|\xi|) p_0(x)
\]
for some $f_0$ homogeneous of order $0$ and $p_0\in C^\infty(M)$, see \eqref{eq:Psihf}.

Now, according to Dave-Haller \cite[Corollary 6(d)]{DH},
\[
  \tau(P) = \int_M p_0(x) \, dx.
\]
On the other hand, a direct calculation of the cocycle \eqref{793} shows that the groupoid residue at a point $x\in M$ is given by 
\[
Res_x P = 
\frac{1}{\log(s)} \left( \log|\delta_s(\xi)|\, p_0(x) - \log|\xi|\, p_0(x) \right)|_{\xi=0}= 
p_0(x),
\]
in an analogous fashion to the proof of Lemma \ref{754}, see Equation \eqref{eq:cocycle_calculation} and following.
The result follows.
\end{proof}

\bibliographystyle{plain}


\bibliography{bibarticle}

\end{document}